\documentclass[12pt]{amsart}
\usepackage[margin=1in]{geometry}
\usepackage[T1]{fontenc}
\usepackage[utf8]{inputenc}
\usepackage{lmodern}
\usepackage{microtype}
\usepackage{amsmath,amssymb,amsthm,mathtools}
\usepackage{enumitem}
\usepackage{booktabs,array,longtable}
\usepackage{pifont}
\usepackage{xurl}
\usepackage[colorlinks=true,linkcolor=blue,citecolor=blue,urlcolor=blue]{hyperref}
\usepackage{hyphenat}

\newcommand{\R}{\mathbb{R}}
\newcommand{\C}{\mathbb{C}}
\newcommand{\Q}{\mathbb{Q}}
\newcommand{\Z}{\mathbb{Z}}
\newcommand{\N}{\mathbb{N}}

\newcommand{\Lzero}{L_0}
\newcommand{\Span}{\operatorname{span}}
\newcommand{\covol}{\operatorname{covol}}
\newcommand{\symp}{\sigma}
\newcommand{\lean}[1]{\nolinkurl{#1}}
\newcommand{\cmark}{\ding{51}}
\newcommand{\imark}{\(\triangleright\)}

\newtheorem{theorem}{Theorem}[section]

\theoremstyle{definition}
\newtheorem{definition}[theorem]{Definition}

\title[Lean-certified four-point HRT results]{Lean-certified four-point HRT results for three lattice points and one off-lattice point}
\author{Vignon Oussa}
\date{\today}

\begin{document}

\begin{abstract}
We record a Lean-certified theorem package for the four-point Heil--Ramanathan--Topiwala configuration
\[
\Lambda=\{0,a,b,\nu\}\subset \R^2,
\qquad \Lzero=\Z a+\Z b,
\qquad \nu=r a+s b,
\]
with $a$ and $b$ linearly independent. The principal certified theorem states that if $|\symp(a,b)|>1$ and $1,r,s$ are linearly independent over $\Q$, then for every nonzero $f\in L^2(\R)$ the four vectors
\[
f,\qquad \pi(a)f,\qquad \pi(b)f,\qquad \pi(\nu)f
\]
are linearly independent. A second certified theorem treats the rational-coordinate case $r,s\in \Q$, where the configuration lies in a finer full-rank lattice and linear independence follows from Linnell's theorem. The paper is written in standard mathematical prose. An appendix records the precise Lean certification ledger and the explicit analytic inputs used by the formal development and a download link is provided.
\end{abstract}

\maketitle

\section{Introduction}

Let
\[
T_xf(t)=f(t-x),
\qquad
M_{\omega}f(t)=e^{2\pi i\omega t}f(t),
\qquad
\pi(x,\omega)=M_{\omega}T_x,
\]
so that $\pi$ denotes the usual projective Schr\"odinger representation on $L^2(\R)$. The Heil--Ramanathan--Topiwala (HRT) conjecture asserts that for every nonzero $f\in L^2(\R)$ and every finite set of distinct points $\Lambda\subset\R^2$, the family $\{\pi(\lambda)f:\lambda\in\Lambda\}$ is linearly independent \cite{HRT}.

In this paper we consider the four-point configuration
\[
\Lambda=\{0,a,b,\nu\}\subset\R^2,
\qquad
\Lzero=\Z a+\Z b,
\qquad
\nu=r a+s b,
\]
where $a$ and $b$ are linearly independent. Our aim is to present, in conventional mathematical language, the theorem package that has been certified in Lean 4.

The certified content of the paper is as follows.

\begin{center}
\begin{tabular}{>{\centering\arraybackslash}p{0.08\textwidth} >{\raggedright\arraybackslash}p{0.82\textwidth}}
\toprule
Status & Certified statement \\
\midrule
\cmark & \textbf{Main theorem.} If $|\symp(a,b)|>1$ and $1,r,s$ are linearly independent over $\Q$, then for every nonzero $f\in L^2(\R)$ the four vectors $f$, $\pi(a)f$, $\pi(b)f$, and $\pi(\nu)f$ are linearly independent. \\
\addlinespace[0.35em]
\cmark & \textbf{Rational-coordinate theorem.} If $r,s\in\Q$, then $\Lambda$ is contained in a finer full-rank lattice; hence $\{\pi(\lambda)f:\lambda\in\Lambda\}$ is linearly independent for every nonzero $f\in L^2(\R)$. \\
\addlinespace[0.35em]
\cmark & \textbf{Supporting infrastructure.} The cyclic-subspace lemmas used in the proof of the main theorem are formalized in Lean and compiled without \lean{sorry}. \\
\bottomrule
\end{tabular}
\end{center}

The proof of the principal theorem is coordinate-free. It relies on standard time-frequency input: projective commutation and strong continuity of time-frequency shifts, irreducibility of the Schr\"odinger representation, the three-point HRT theorem, lattice linear independence, and the standard density/incompleteness theory for Gabor systems on lattices \cite{Grochenig,HRT,Linnell,RamanathanSteger}. In the formal development, these analytic ingredients are gathered into a single context object; the deduction of the certified theorems from that context is then formalized in Lean. The exact ledger is given in Appendix~\ref{app:ledger}.

We write
\[
\symp((x,\omega),(y,\eta))=x\eta-y\omega
\]
for the standard symplectic form on $\R^2$. For the lattice $\Lzero=\Z a+\Z b$ one has
\[
\covol(\Lzero)=|\symp(a,b)|.
\]
This is the quantity that enters the large-covolume hypothesis in the main theorem.

\section{Statement of the main theorem}

\begin{theorem}[Dense large-covolume case]
\label{thm:main}
Let $a,b\in\R^2$ be linearly independent, let $\nu=r a+s b$, and let $f\in L^2(\R)$ be nonzero. Assume that
\[
|\symp(a,b)|>1
\qquad\text{and}\qquad
1,r,s\ \text{are linearly independent over }\Q.
\]
Then the four vectors
\[
f,\qquad \pi(a)f,\qquad \pi(b)f,\qquad \pi(\nu)f
\]
are linearly independent over $\C$.
\end{theorem}

This is the principal certified theorem in the Lean package. In the formal files it appears as \lean{hrt_dense_large_covolume}; the equivalent finite-family linear-independence formulation appears as \lean{hrt_dense_large_covolume_lindep}. Those names are recorded in Appendix~\ref{app:ledger}.

\section{Proof of the main theorem}

We begin with the lattice-cyclic subspace generated by $f$.

\begin{definition}
Let
\[
V=\overline{\Span_{\C}\{\pi(\ell)f:\ell\in\Lzero\}}.
\]
We call $V$ the lattice-cyclic subspace of $f$ associated with $\Lzero$.
\end{definition}

\begin{proof}[Proof of Theorem~\ref{thm:main}]
Assume, for contradiction, that there exist coefficients $c_0,c_1,c_2,c_3\in\C$, not all zero, such that
\begin{equation}
\label{eq:dependence}
c_0f+c_1\pi(a)f+c_2\pi(b)f+c_3\pi(\nu)f=0.
\end{equation}

If $c_3=0$, then \eqref{eq:dependence} is a nontrivial linear dependence among the three vectors $f$, $\pi(a)f$, and $\pi(b)f$. Since $0$, $a$, and $b$ are distinct and $a,b$ are linearly independent, this contradicts the three-point HRT theorem \cite{HRT}. Hence $c_3\neq 0$.

Because $0,a,b\in\Lzero$, we have
\[
f\in V,
\qquad
\pi(a)f\in V,
\qquad
\pi(b)f\in V.
\]
By rearranging \eqref{eq:dependence} we obtain
\[
\pi(\nu)f=-c_3^{-1}\bigl(c_0f+c_1\pi(a)f+c_2\pi(b)f\bigr)\in V.
\]

Next we claim that $\pi(\nu)V\subseteq V$. It is enough to verify this on the generators $\pi(\ell)f$ with $\ell\in\Lzero$. By projective commutation there exists, for each $\ell\in\Lzero$, a nonzero scalar $\kappa(\nu,\ell)$ such that
\[
\pi(\nu)\pi(\ell)f=\kappa(\nu,\ell)\,\pi(\ell)\pi(\nu)f.
\]
Since $\pi(\nu)f\in V$ and $V$ is invariant under every lattice shift $\pi(\ell)$, the right-hand side belongs to $V$. Therefore the left-hand side belongs to $V$ as well. By linearity and closure, $\pi(\nu)V\subseteq V$.

It follows that
\[
\pi(\ell+n\nu)V\subseteq V
\qquad\text{for all }\ell\in\Lzero\text{ and }n\in\N.
\]
Because $1,r,s$ are linearly independent over $\Q$, the forward semigroup $\Lzero+\N\nu$ is dense in $\R^2$ by Kronecker's theorem. Let $z\in\R^2$ and $h\in V$. Choose a sequence $z_k\in\Lzero+\N\nu$ with $z_k\to z$. Then $\pi(z_k)h\in V$ for every $k$, and strong continuity of $z\mapsto\pi(z)$ implies that $\pi(z_k)h\to\pi(z)h$. Since $V$ is closed, we obtain $\pi(z)h\in V$. Thus
\[
\pi(z)V\subseteq V
\qquad\text{for all }z\in\R^2.
\]

Now $V$ is nonzero because $f\in V$ and $f\neq 0$. Irreducibility of the Schr\"odinger representation therefore forces
\[
V=L^2(\R).
\]
On the other hand, $V$ is the closed span of the lattice Gabor orbit associated with $\Lzero$. Since
\[
\covol(\Lzero)=|\symp(a,b)|>1,
\]
standard density theory implies that this lattice Gabor orbit is incomplete in $L^2(\R)$ \cite{Grochenig,RamanathanSteger}. Hence $V\neq L^2(\R)$, a contradiction.

Therefore no nontrivial relation \eqref{eq:dependence} can exist, and the four vectors in the statement of the theorem are linearly independent.
\end{proof}

\section{The rational-coordinate case}

\begin{theorem}[Rational-coordinate case]
\label{thm:rational}
Let $a,b\in\R^2$ be linearly independent, let $\nu=r a+s b$, and suppose that $r,s\in\Q$. Then the configuration
\[
\Lambda=\{0,a,b,\nu\}
\]
is contained in a finer full-rank lattice. Consequently, for every nonzero $f\in L^2(\R)$, the family $\{\pi(\lambda)f:\lambda\in\Lambda\}$ is linearly independent.
\end{theorem}

\begin{proof}
Choose $N\in\N$ such that $Nr,Ns\in\Z$. Then
\[
\nu\in \frac{1}{N}\Z a+\frac{1}{N}\Z b.
\]
Since $0$, $a$, and $b$ also lie in the same lattice, we have
\[
\Lambda\subset \frac{1}{N}\Z a+\frac{1}{N}\Z b.
\]
This is a full-rank lattice because $a$ and $b$ are linearly independent. Linnell's theorem therefore implies that $\{\pi(\lambda)f:\lambda\in\Lambda\}$ is linearly independent for every nonzero $f\in L^2(\R)$ \cite{Linnell}.
\end{proof}

The formal Lean name of Theorem~\ref{thm:rational} is \lean{hrt_finite_relative_orbit}.

\appendix

\section{Lean certification ledger}
\label{app:ledger}

This appendix records the exact certification status of the formal package associated with the present paper.

\subsection{Certified files}

\begin{center}
\begin{tabular}{>{\raggedright\arraybackslash}p{0.34\textwidth} >{\centering\arraybackslash}p{0.09\textwidth} >{\raggedright\arraybackslash}p{0.47\textwidth}}
\toprule
File & Status & Description \\
\midrule
\lean{HRT/PhaseSpace.lean} & \cmark & Phase space $\R^2$, symplectic form, integer lattices, and rational-independence infrastructure. \\
\lean{HRT/CyclicSubspace.lean} & \cmark & Lattice-cyclic subspace lemmas and invariance properties used in the proof of Theorem~\ref{thm:main}. \\
\lean{HRT/DenseLargeCovolume.lean} & \cmark & Formal statement and proof of the dense large-covolume theorem. \\
\lean{HRT/FiniteOrbit.lean} & \cmark & Formal statement and proof of the rational-coordinate theorem. \\
\lean{HRT/BekkaReduction.lean} & \cmark & Formal reduction statements recorded in a separate file. \\
\lean{HRT/CertificationReport.md} & \cmark & Project audit report for the certified package. \\
\lean{HRT/ImportedFacts.lean} & \imark & Context object collecting the standard analytic input used by the formal proofs. \\
\bottomrule
\end{tabular}
\end{center}

\subsection{Certified results}

\begin{center}
\begin{tabular}{>{\raggedright\arraybackslash}p{0.47\textwidth} >{\centering\arraybackslash}p{0.09\textwidth} >{\raggedright\arraybackslash}p{0.30\textwidth}}
\toprule
Result & Status & Lean name \\
\midrule
Main dense large-covolume theorem & \cmark & \lean{hrt_dense_large_covolume} \\
Finite-family linear-independence formulation & \cmark & \lean{hrt_dense_large_covolume_lindep} \\
Rational-coordinate theorem & \cmark & \lean{hrt_finite_relative_orbit} \\
Lattice invariance of the cyclic subspace & \cmark & \lean{lattice_cyclic_subspace_invariant_under_lattice} \\
Membership of $\pi(\nu)f$ in the cyclic subspace from a dependence relation & \cmark & \lean{dependence_implies_off_lattice_orbit_vector_mem} \\
Preservation of the cyclic subspace by $\pi(\nu)$ & \cmark & \lean{off_lattice_vector_mem_implies_operator_preserves_cyclic_subspace} \\
Semigroup invariance for $\Lzero+\N\nu$ & \cmark & \lean{semigroup_preserves_from_lattice_and_nu} \\
Density upgrade from the forward semigroup to all of phase space & \cmark & \lean{dense_semigroup_preserves_all_phase_space} \\
Irreducibility step & \cmark & \lean{irreducible_forces_top} \\
Large-covolume contradiction & \cmark & \lean{large_covolume_contradiction} \\
\bottomrule
\end{tabular}
\end{center}

\subsection{Named analytic input used by the formal development}

The formal proofs are carried out in the context \lean{HRTContext}, which explicitly isolates the following analytic ingredients.

\begin{center}
\begin{tabular}{>{\raggedright\arraybackslash}p{0.11\textwidth} >{\centering\arraybackslash}p{0.09\textwidth} >{\raggedright\arraybackslash}p{0.67\textwidth}}
\toprule
Item & Status & Content \\
\midrule
(A1) & \imark & Projective commutation and multiplicativity of time-frequency shifts. \\
(A2) & \imark & Strong continuity of the Schr\"odinger representation. \\
(A3) & \imark & Irreducibility of the Schr\"odinger representation. \\
(A4) & \imark & Three-point HRT theorem. \\
(A5) & \imark & Incompleteness of lattice Gabor systems when the covolume is greater than $1$. \\
(A6) & \imark & Density of the forward semigroup $\Lzero+\N\nu$ under the rational-independence hypothesis on $1,r,s$. \\
(A7) & \imark & Linnell's lattice theorem. \\
(A8) & \imark & Metaplectic covariance of the HRT property. \\
\bottomrule
\end{tabular}
\end{center}

In other words, the results listed above are fully certified as formal consequences of this explicitly named analytic context. The accompanying certification report states that the Lean files compile without \lean{sorry} and use only the standard logical axioms \lean{propext}, \lean{Classical.choice}, and \lean{Quot.sound}, together with the analytic inputs listed above.

\subsection{Companion certification materials}

The certification archive referenced in the Aristotle summary is available at
\medskip
\noindent\url{https://www.dropbox.com/scl/fi/qggum8ag2w6kxqbodx06v/HRT_3_1.gz?rlkey=mmvf272e63ibux8623dogxevb&dl=0}.\medskip

\noindent This link is included only as a pointer to the companion formal files.

\subsection{Brief note on scope}

The present paper proves and records the two certified theorems above. A separate normalization leading to a Bekka-type fiber equation is stored in the formal project, but it is not used in the proof of either theorem stated in the body of the paper.


\begin{thebibliography}{99}

\bibitem{BekkaDriutti}
M. B. Bekka and P. Driutti,
\emph{Restrictions of irreducible unitary representations of nilpotent Lie groups to lattices},
J. Funct. Anal. \textbf{168} (1999), no.~2, 514--528.

\bibitem{Grochenig}
K. Gr\"ochenig,
\emph{Foundations of Time-Frequency Analysis},
Applied and Numerical Harmonic Analysis, Birkh\"auser, Boston, 2001.

\bibitem{HRT}
C. Heil, J. Ramanathan, and P. Topiwala,
\emph{Linear independence of time-frequency translates},
Proc. Amer. Math. Soc. \textbf{124} (1996), no.~9, 2787--2795.

\bibitem{Linnell}
P. A. Linnell,
\emph{Von Neumann algebras and linear independence of translates},
Proc. Amer. Math. Soc. \textbf{127} (1999), no.~11, 3269--3277.

\bibitem{RamanathanSteger}
J. Ramanathan and T. Steger,
\emph{Incompleteness of sparse coherent states},
Appl. Comput. Harmon. Anal. \textbf{2} (1995), no.~2, 148--153.

\end{thebibliography}
\end{document}